\documentclass[11pt,reqno]{amsart}

 \usepackage{mathrsfs}

\usepackage{color}
\usepackage{amsmath, amsthm, amssymb}
\usepackage{amsfonts}
\usepackage[dvips]{epsfig}
\usepackage{graphicx}
\usepackage{caption}
\usepackage{subcaption}
\usepackage[english]{babel}
\usepackage{hyperref}
\usepackage{textcomp}
\usepackage{tikz}
\usepackage{rotating}
\usepackage[utf8]{inputenc}
\usepackage{cite}
\usepackage{amscd}
\usepackage{color}
\usepackage{bm}
\usepackage{enumerate}
\usepackage{amsmath}
\usepackage{bbm}
\usepackage{verbatim}
\usepackage{hyperref}
\usepackage{amstext}
\usepackage{latexsym}
\theoremstyle{plain}
\newtheorem{theorem}{Theorem}[section]

\newtheorem{proposition}[theorem]{Proposition}
\newtheorem{lemma}[theorem]{Lemma}

\numberwithin{theorem}{section}
\numberwithin{equation}{section}

\newcommand{\average}{{\mathchoice {\kern1ex\vcenter{\hrule height.4pt
width 6pt depth0pt} \kern-9.7pt} {\kern1ex\vcenter{\hrule
height.4pt width 4.3pt depth0pt} \kern-7pt} {} {} }}

\def\R{\mathbb{R}}

\def\div{\text{div}}

\renewcommand{\b }{\beta }
\renewcommand{\d}{\delta }

\newcommand{\D }{\Delta }

\newcommand{\e }{\varepsilon }
\newcommand{\g }{\gamma}

\newcommand{\G }{\Gamma}

\newcommand{\n }{\nabla }

\renewcommand{\phi}{\varphi}

\newcommand{\s }{\sigma }

\renewcommand{\t }{\tau }

\renewcommand{\o }{\omega }
\renewcommand{\O }{\Omega }

\newcommand{\ov}{\overline}

\newcommand{\be}{\begin{equation}}
\newcommand{\ee}{\end{equation}}

\newcommand{\de}{\partial}

\newcommand{\ti}{\widetilde}

\newcommand{\K }{\mathcal{K}}

\newcommand{\calC }{\mathcal{C}}

\newcommand{\cK}{{\mathcal K}}

\newcommand{\B}{{Q}}

\newcommand{\1}{\mathbbm{1}}

\renewcommand{\epsilon}{\varepsilon}

\begin{document}
\title[Nonlocal diffusion of smooth sets]{Nonlocal diffusion of smooth sets}

\author{Anoumou Attiogbe}
\address{A. A. : African Institute for Mathematical Sciences in Senegal, KM 2, Route de
Joal, B.P. 14 18. Mbour, Senegal.}
\email{anoumou.attiogbe@aims-senegal.org}

\author{Mouhamed Moustapha Fall}
\address{M. M. F.: African Institute for Mathematical Sciences in Senegal, KM 2, Route de
Joal, B.P. 14 18. Mbour, Senegal.}
\email{mouhamed.m.fall@aims-senegal.org}
\author{El Hadji Abdoulaye THIAM}
\address{H. E. A. T. : Université de Thies, UFR des Sciences et Techniques, département de mathématiques, Thies.}
\email{elhadjiabdoulaye.thiam@univ-thies.sn}


\begin{abstract}
We consider  normal velocity of smooth sets evolving by the $s-$fractional diffusion. We prove that for small time, the normal velocity of  such sets is nearly proportional to the   mean curvature of the boundary of the initial set for $s\in [\frac{1}{2}, 1)$ while,   for $s\in (0, \frac{1}{2})$,  it is nearly proportional to the fractional mean curvature of the initial set. Our results shows that  the motion by (fractional) mean curvature flow can be approximated by   fractional heat diffusion and  by a  diffusion by means of harmonic extension  of smooth sets.
%
%
%
%In the spirit of Bence-Merriman-Osher \cite{BMO}, Caffarelli-Souganidis \cite{CPS} provided a numerical methods to compute (fractional) mean curvature flow, in terms of the fractional heat equation. In this paper, we  use the level-set approach to provide an analytic proof of the Caffarelli-Souganidis algorithm.
\end{abstract}
\maketitle
\date{today}
%
%
%\tableofcontents
%
\section{Introduction}\label{s:Intr}
For $N \geq 2$, we let $\O_0$ be a bounded open set of $\R^N$ with boundary $\Gamma_0$. Consider the heat equation with initial data the indicator function of the set $\O_0$:
\begin{equation}\label{HeatEquation}
\begin{cases}
\displaystyle \frac{\de u}{\de t}-\D u=0 \qquad &\textrm{in } \R^N \times (0, t_1]\\\\
u(x, 0)= \1_{\O_0}(x) \qquad &\textrm{ on } \R^N.
\end{cases}
\end{equation}
for some time $t_1>0$. 
In 1992, Bence-Merriman-Osher \cite{BMO} provided a computational algorithm for tracking the evolution in time of the set $\O_0$ whose boundary $\Gamma_0$ moves with normal velocity proportional to its classical mean curvature. At time $t_1>0$, they considered 
$$
\O_1= \lbrace x\in \R^N: \quad u(x, t_1) \geq {1}/{2} \rbrace.
$$
Bence-Merriman-Osher \cite{BMO} applied iteratively this procedure to generate a sequence of sets $\left(\O_j\right)_{j\geq 0}$ and conjectured in \cite{BMO} that their boundaries $\Gamma_j$ evolved by mean curvature flow. 
Later Evans \cite{Evans} provided a rigorous proof for the Bence-Merriman-Osher algorithm by means of  the level-set approach to mean curvature flow developed by Osher-Sethian \cite{OS}, Evans-Spruck \cite{ES1, ES2, ES3, ES4} and Chen-Giga-Goto \cite{CGG}. For related works in this direction, we refer the reader to   \cite{BG,Ishii,IPS,Vivier,Leoni,GIO,swartz} and references therein.\\

Recently Caffarelli and Souganadis considered in \cite{CPS}  nonlocal diffusion of open sets $E\subset \R^N$ given by 
%
%problem \eqref{HeatEquation}, replacing the diffusion operator $\D$ with the nonlocal operator $(-\D)^s$, for $s\in (0,1)$ and $1_{\O_0}$ with  $\tau_E(x)= \1_E(x)-\1_{E^c}(x).$
%
 % They  proved that, after an appropriate scaling, the  scaled threshold dynamics-type algorithms corresponding to the fractional Laplacian $(-\D)^s$, with $s\in (0,1)$, converges to moving fronts: for $s\geq \frac{1}{2}$, the resulting interface moves by mean curvature, while for $s<\frac{1}{2}$, it moves by  fractional mean curvature flow. In particular, given $E$ an open subset of $\R^N$, with boundary $\de E$ containing the origin, we let $\nu=\nu(0)$ denote the outward unit normal to $\partial E$ at $0$. Then consider the fractional heat equation:
\begin{align}\label{FractHeatEq}
\begin{cases}
\displaystyle \frac{\de u}{\de t}+ \left(-\D\right)^s u=0 \qquad &\textrm{in } \R^N \times \left(0, \infty\right)\\
u(x, 0)= \tau_E(x) &\textrm{in } \R^N \times \lbrace t=0 \rbrace,
\end{cases}
\end{align}
where
$$
\tau_E(x)= \1_E(x)-\1_{\R^N\setminus \ov E}(x).
$$
We consider the fractional  heat kernel $K_s$  with Fourier transform given by $\widehat K(\xi, t)=e^{-t|\xi|^{2s}}$. It satisfies
$$
\begin{cases}
\displaystyle \frac{\de K_s}{\de t}+ \left(-\D\right)^s K_s=0 \qquad &\textrm{ in } \R^N \times \left(0, \infty\right)\\
K_s= \d_0 &\textrm{ on } \R^N \times \lbrace t=0 \rbrace.
\end{cases}
$$
 It follows that the unique bounded solution to     \eqref{FractHeatEq} is given by  
\be \label{eq:uxt}
u(x,t)=K_s(\cdot, t)\star \t_E (x)=\int_{\R^N} K_s(x-y, t) \t_E(y) dy.
\ee
By solving a finite number of times \eqref{FractHeatEq} for a small fixed time step $\s_s(h)$,  the authors in  \cite{CPS}    find a discrete  family of sets 
 $$
E^h_0=E, \qquad E^h_{nh}= \lbrace x\in \R^N: \quad K_s (\cdot,\s_s(h))\star\t_{E^h_{(n-1)h}}(x)>0 \rbrace ,
$$
for a suitable scaling function $\s_s$ to be defined below. It is proved in \cite{CPS} that  as  $nh\to t$,   $  \de E^h_{nh}$ converges, in a suitable sense, to $\G_t$.  Here,   the family of hypersurface $\{\G_t\}_{t>0}$, with $\G_0=\de E$,  evolves under generalized   mean curvature flow for $s\in [\frac{1}{2}, 1)$ and under generalized fractional mean curvature flow   for    $s \in (0,\frac{1}{2})$. We refer the reader to  \cite{CPS,Imbert,Evans} for the notion generalized  (nonlocal) mean curvature flow which considers the level sets of viscosity solutions to  quasilinear  parabolic integro-differential equations.\\

In the present paper,  we are interested in the normal velocity of the sets 
\be \label{eq:scaled-sets}
E_t:=\left\{x\in \R^N\,:\,   \K_s (\cdot,\s_s(t))\star\t_{E}(x)>0  \right\}
\ee
as they depart from a sufficiently smooth   initial set $E_0:=E$.  We consider  here and in the following 
\be \label{KernelFormula}
\cK_s(x,t)=t^{-\frac{N}{2s}}P_s(t^{-\frac{1}{2s}}x), \qquad \textrm{ for some radially symmetric function   $P_s\in C^1(\R^N)$.}
\ee
We make the following assumptions: 
 \begin{equation}\label{KernelEstimate}
\frac{\calC^{-1}_{N,s}}{1+|y|^{N+2s}} \leq P_s(y) \leq \frac{\calC_{N,s}}{1+|y|^{N+2s}},   \qquad |\nabla P_s(y)| \leq \frac{\calC_{N,s}}{1+|y|^{N+2s+1}}.
\end{equation}
and
\begin{equation}\label{ConvergenceUnifandLoc}
\lim_{t\to 0} t^{-1} \K_s(y, t)=\frac{C_{N,s}}{|y|^{N+2s}} \qquad\textrm{ locally uniformly in $\R^N\setminus \{0\}$,}
\end{equation}
for some constants $C_{N,s}, \calC_{N,s}>0 $. In the Section \ref{eq:sss-Appli} below, we provide examples of valuable kernels $\cK_s$ satisfying the above properties.\\
Now, as we shall see below (Lemma \ref{Lemma21}), for  $t>0$ small, $ \nabla_xu(x,\s_s(t)) \not=0$ for all $x\in B(y, \s_s(t)^{\frac{1}{2s}})$ and $y\in \de E$.  Hence $\partial E_t$ is a 
$C^1$ hypersurface, for small $t>0$.   For  $t>0$ and $y\in \de E$, we let $v=v(t,y)$ be such that 
\be \label{eq:def-vtnu}
y+vt\nu(y) \in \de E_t\cap B(y, \s_s(t)^{\frac{1}{2s}})  ,
\ee
where $\nu(y) $ is the unit exterior normal of $E$ at $y$.  In the spirit of  the work of Evans \cite{Evans} on  diffusion of smooth sets, we provide in this paper an expansion of $v(t,y)$ as $t\to 0$. It turns out that $v(0,y)$ is proportional to the fractional mean curvature of $\de E$ at $y$ for $s\in (0,1/2)$ and $v(0,y)$ is proportional to the classical mean curvature of $\de E$ at $y$ for $s\in [1/2,1)$. \\
 We notice that it is not a priori clear from \eqref{eq:def-vtnu}, that  $v$ remains finite as $t\to0$. This is where the (unique)  appropriate choice of  $\s_s(t)$ enters during our estimates.  Here and in the following, we define 
\begin{equation}\label{scaling}
\sigma_s(t)=
\begin{cases}
t^{\frac{2s}{1+2s}} &\qquad \textrm{for } s\in(0, 1/2),\\
t^{s} &\qquad \textrm{for } s\in(1/2, 1)
\end{cases}
\end{equation}
and for $s=1/2$, $\s_s(t)$ is the unique positive solution to 
\begin{equation}\label{scaling1}
t=\s_{1/2}^2(t)|\log(\s_{1/2}(t))|.
\end{equation} 
Before stating our main result, we recall that for $s \in (0, \frac{1}{2})$ and $\de E$ is of class $\calC^{1, \b}$ for some $\b >2s$,   the fractional mean curvature of $\de E$ is   defined for    $x\in \de E$ as
$$
H_s(x):=P. V. \int_{\R^N} \frac{\t_E(y)}{|x-y|^{N+2s}} dy= \lim_{\e \to 0} \int_{\R^N \setminus B_\e(x)} \frac{\t_E(y)}{|x-y|^{N+2s}} dy.
$$
On the other hand, if $\de E$ is of class $\calC^{2}$ then the normalized mean curvature of $\de E$ is given, for    $x\in \de E$,  by 
$$
H(x): =   \frac{N-1}{N+1}  \lim_{\e \to 0}\frac{1}{\e |B_\e(x)|} \int_{   B_\e(x)}\! \t_E(y) \,dy,
$$
see also \eqref{eq:def-mc} and \cite{Fall23}.
Having fixed the above definitions, we now state our main result.
\begin{theorem}\label{MainResult}
We let $s\in (0,1)$ and $E \subset\R^N$, $N\geq 2$. We assume, for $s\in (0,1/2)$, that  $\de E$ is of class $\calC^{1,\b}$ for some $\b>2s$   and that  $\de E$ is of class $\calC^{3}$, for $s \in[{1}/{2}, 1)$.
Then, as $t\to 0$,  the expansion of $v(t,y)$, defined in \eqref{eq:def-vtnu},  is given, locally uniformly in $y\in \de E$,  by 
\begin{equation*}
v(t,y)=\begin{cases}
a_{N,s} H_s(y)+o_t(1)&\qquad \textrm{for } s \in(0,1/2)\\\\
 b_{N}(t)H(y)+ O\left(\frac{1}{\log(\s_{1/2}(t))}\right)&\qquad \textrm{for } s=1/2\\\\
 c_{N,s}H(y)+O\left(t^{\frac{2s-1}{2}}\right)&\qquad \textrm{for } s \in(1/2, 1),
\end{cases}
\end{equation*}
where $H_s$ and  $H$ are respectively the fractional and the classical mean curvatures of  $\de E$  and the positive constants $a_{N,s}$, $b_{N,1/2}$ and $c_{N,s}$ are given by
$$
a_{N,s}=\frac{C_{N,s} }{\displaystyle2 \int_{\R^{N-1}} P_s\left(y^\prime,0\right) dy^\prime},\quad 
b_{N}(t)= \frac{\displaystyle\int_{B^{N-1}_{\s_{\frac{1}{2}}(t)^{-1}}}|y^\prime|^2  P_{1/2}(y^\prime,0) dy^\prime}{2|\log(\s_{\frac{1}{2}}(t))|\displaystyle\int_{\R^{N-1}}  P_{1/2}(y^\prime,0) dy'} ,
%
%\frac{|S^{N-2}| C_{N, 1/2}}{2 \displaystyle\int_{\R^{N-1}}  P_{1/2}(y^\prime,0) dy^\prime},
%
 \quad   c_{N,s}=\frac{\displaystyle\int_{\R^{N-1}}|y^\prime|^2  P_s(y^\prime,0) dy^\prime}{\displaystyle 2\int_{\R^{N-1}}  P_s(y^\prime,0) dy^\prime}
$$
and $P_s(y):=\K_s(y,1)$.
\end{theorem}
Some remarks are in order.
The assumption of $E$ being of class $\calC^3$ in Theorem  \ref{MainResult} is motivated by the result of Evans in \cite{Evans}, where in the case $s=1$ and $\cK_1$, the heat kernel, he obtained  $v=(N-1)H(0)+O(t^{\frac{1}{2}})$. We  notice that from our argument below, we cannot improve  the error term $o_t(1)$ in the case $s\in (0,1/2)$ even if $E$ is of class $\calC^\infty$. This is due to the definition of the fractional mean curvature $H_s$  as a principal value integral.   We finally remark,  in the particular case, that  $ 
\K_{1/2}(y, t)=C_{N}\frac{t}{(t^2+|y|^2)^{\frac{N+1}{2}}},$  we have that
$$
b_N(t)=\frac{|S^{N-2}| C_{N}}{2 \displaystyle\int_{\R^{N-1}}  P_{1/2}(y^\prime,0) dy^\prime}+ O\left(\frac{1}{\log(\s_{1/2}(t))}\right).
$$
\subsection{Some applications of Theorem \ref{MainResult}}\label{eq:sss-Appli}
We next put emphasis on two valuable examples where Theorem \ref{MainResult} applies. 
\begin{itemize}
\item[1)] \textit{Fractional heat  diffusion of smooth sets.}
We recall, see e.g. \cite{BGet, Juan},  that the fractional heat kernel  $K_s$ satisfies \eqref{KernelFormula}, \eqref{KernelEstimate} and \eqref{ConvergenceUnifandLoc} with \begin{equation}\label{ConvergenceUnifandLoc1}
C_{N,s}=s 2^{2s}\frac{sin(s\pi) \G\left(\frac{N}{2}+s\right) \G(s)}{\pi^{1+\frac{N}{2}}}.
\end{equation}
We recall that $K_s$ is known explicitly only in the case $s=1/2$, where  $ 
K_{1/2}(y, t)=C_{N,1/2}\frac{t}{(t^2+|y|^2)^{\frac{N+1}{2}}}.$
In this case Theorem \ref{MainResult} provides an approximation of the (fractional) mean curvature motion by fractional heat diffusion of smooths sets, thereby  extending, in the fractional setting,  Evan's result in \cite{Evans} on heat diffusion of smooth sets.
\item[2)] \textit{Diffusion of smooth sets by Harmonic extension.}
We consider the Poisson kernel on the half space $\R^{N+1}_+: =\R^N\times (0, \infty)$, given by 
\be
 \ov K_s(x,t):=t^{-N}P_s(x/t), \qquad      P_s(x)=\frac{p_{N,s}}{(1+|x|^2)^{\frac{N+2s}{2}}},
\ee
where $p_{N,s}:=\frac{1}{\int_{\R^N} (1+|y|^2)^{-\frac{N+2s}{2}}\,dy}$.
Thanks to the result of Caffarelli and Silvestre in   \cite{CSilv}, the function 
$$
w(x, t)= \ov K_s(\cdot,t)\star\t_E(x)=p_{N,s}t^{2s} \int_{\R^N} \frac{\t_E(y)}{\left(t^2+|y-x|^2\right)^{\frac{N+2s}{2}}} dy
$$
solves 
\begin{align*}
\begin{cases}
\div(t^{1-2s} \n w)=0& \qquad \textrm{in }\R^N\times (0, \infty)\\
w=\t_E& \qquad \textrm{on }\R^N\times \lbrace t=0 \rbrace.
\end{cases}
\end{align*}
It is clear that  $\cK_s(x,t):=\ov K_s(x,t^{\frac{1}{2s}})$  satisfies  \eqref{KernelFormula}, \eqref{KernelEstimate} and \eqref{ConvergenceUnifandLoc} with $C_{N,s}=p_{N,s}$. Hence, Theorem \ref{MainResult} provides  an expansion of the normal velocities  of the boundary of the  sets
$$
 E_t:=\left\{ x\in \R^N\,:\,  \ov K_s(\cdot, \s_s(t)^{\frac{1}{2s}})\star \t_E(x)>0   \right\},
$$
where $\s_s(t)$ is given by \eqref{scaling} and   \eqref{scaling1}. Therefore this Harmonic extension yields an approximation of (fractional) mean curvature motion of smooth sets.
\end{itemize}
%

%
%
%
%Then he studied a level set formulation based on viscocity solution and proved that the geometric flow of a given set $\O$ is a family $\lbrace \O_t\rbrace_{t>0}$ such that the velocity of a point $x\in \de \O_t$ along its outer normal depends on the fractional mean curvature at $x$.\\
%%
%\vspace{3cm}
%
%
%
%
%
%
%
%
We  conclude Section \ref{s:Intr} by noting that the  notion of nonlocal curvature appeared for the first time in   \cite{CPS}.  Later on, the  study of  geometric problems involving fractional mean curvature has attracted a lot of interest, see \cite{CRS,  CRS, AVal}, the survey paper \cite{Fall23} and the references therein.   While the   mean curvature flow  is well studied, see e.g. \cite{AG,EH,GH,Huisken, Giga,GA}, its  fractional counterpart    appeared only recently in the literature, see e.g. \cite{CSV,SV,CMP1,CMP2,CMP3,Imbert,LD}.\\

We finally remark that the changes  of normal velocity of the nonlocal diffused   sets as $s$ varies in (0,1/2) and [1/2,1), appeared analogously   in phases  transition problems, see e.g. \cite{SV,Garroni1,Garroni2}.\\

\bigskip

\noindent
   \textbf{ Acknowledgments} \\
 This work is supported by the Alexander von Humboldt foundation and the    German Academic
Exchange Service (DAAD). Part of this work was done while the authors were visiting the International Center for Theoretical Physics (ICTP) in December 2019 within the  Simons associateship program.
%
%
%
%
%
%The paper is organised as follows. Section \ref{section4} is devoted to some preliminary results. In Section \ref{section1}, we study the case $s\in (0, \frac{1}{2})$, in Section \ref{section2}, we deal with the case $s\in (\frac{1}{2}, 1)$ and in Section \ref{section3}, we complete the proof of our main result by studying the case $s=\frac{1}{2}$.
%
%
%
%
%
%
\section{Preliminary results and notations}\label{section4}
%
%
%Moreover, the Nonlocal Mean Curvature enjoys a geometric expression that can be derived via integration by parts. That is, provided $\int_{\Sigma_0} \frac{d\sigma(y)}{(1+|y|)^{N+2s-1}}<\infty$, we have
%$$
%H_s(x)=-\frac{C_{N,s}}{s} \int_{\Sigma_0} \frac{(x-y)\cdot \nu_{\Sigma_0}(y)}{|x-y|^{N+2s}} d\s(y),
%$$
%where $\nu_{\Sigma_0}(y)$ denotes the outer unit normal to $\Sigma_0$ at $y$.\\
Unless otherwise stated,  we assume for the following that $E$ is an open set of class $\calC^{1, \b}$, with  $0\in \de E$ and the unit normal of $\de E$ at 0 coincides with  $e_N$.
We denote by $Q_r=B^{N-1}_r \times (-r, r)$ the cylinder of $\R^N$ centred at the origin with $B^{N-1}_r$ the ball of $\R^{N-1}$ centred at the origin with radius $r>0$. Decreasing $r$, if necessary,  we may assume  that
\be\label{eq:deE-graph}
  E\cap Q_r=\lbrace(y',y_N)\in B_r^{N-1}\times \R\,:\, y_N> \gamma(y^\prime) \rbrace, 
\ee
with $\g\in \calC^{1, \b}(B_r^{N-1})$ satisfying
\begin{equation}\label{ParametrisationNonLocal}
\gamma(y^\prime)=O\left(|y^\prime|^{1+\b}\right).
\end{equation}
In the  following, for $f, g: \R \to \R$, we write $g(t):=O(f(t))$ if
$$
|g(t)| \leq C|f(t)|.
$$
We also write  $g(t)=o(f(t))$ if $g(t)=O(f(t))$ and moreover when $f(t) \neq 0 $, we have
$$
\lim_{t \to 0}\frac{|g(t)|}{|f(t)|}=0.
$$
We denote by $o_t(1)$ any function that tends to zero when $t \to 0$.
If in addition, $\de E$ is of class $\calC^3$, then for $y'\in B_r^{N-1}$, we have
\begin{equation}\label{star}
  \gamma(y^\prime)=\frac{1}{2}D^2\gamma(0) [y', y']+O\left(|y^\prime|^{3}\right)
\end{equation}
and the normalized mean curvature of $\de E$ at 0 is  given by
\be \label{eq:def-mc}
	H(0)=\frac{\D \gamma(0)}{N-1}=  \frac{N-1}{N+1}  \lim_{\e \to 0}\frac{1}{\e |Q_\e(x)|} \int_{   Q_\e(x)}\! \t_E(y) \,dy.
\ee
Recall that the unit exterior normal $\nu(y'):=\nu(y',\g(y'))$ of $E$ and the volume element $d\s(y')$ on $\de E\cap Q_r$ are given by 
	\begin{equation}\label{nu}
	\nu(y')=\frac{(-\nabla\gamma(y'),1)}{\sqrt{1+|\nabla\gamma(y')|^2}} \qquad
	\textrm{and} \qquad d \sigma(y')=\sqrt{1+|\nabla\gamma(y')|^2} dy^\prime.
	\end{equation}
We finally note,  in view of \eqref{KernelFormula} and \eqref{KernelEstimate}, that  we have 
\begin{equation}\label{KernelFormula-1}
0<  \K_s(y,t) \leq C  \frac{t}{|y|^{N+2s}} \qquad\textrm{ for all $y \in \R^N\setminus \{0\}$, $t>0$} ,
\end{equation}
for some positive constant $C=C(N,s)$.
We start with the following result.
%\section{regularity of the level sets}
\begin{lemma}\label{Lemma21}
Let  $s\in (0,1)$ and $E$ be a $\calC^{1,\b}$ hypersurface satisfying \eqref{eq:deE-graph}. Define
$$
w(z, t)=\int_{\R^N} \K_s(z-y,t) \t_E(y) dy.
$$
Then there exist $t_0,C>0$, only depending on $N,s,\b$ and $E$,  such that for all $t \in (0, t_0)$ and $z\in B_{t^{\frac{1}{2s}}}$,
\be \label{eq:nu_x}
 \frac{\de   w}{\de z_N} (z, t)\geq C t^{-1/2s}. 
\ee 
As a consequence, for all $t \in (0, t_0)$, the set 
\begin{equation}\label{Partie2}
\lbrace 
z\in \R^N\,:\, w(z, t)=0
\rbrace \cap B_{t^{\frac{1}{2s}}} \quad \textrm{is of class $C^1$.}
\end{equation}
\end{lemma}
\begin{proof}
We fix  $t>0$ small so that $t^{\frac{1}{2s}}<\frac{r}{8}$ and  let $z\in B_{t^{\frac{1}{2s}}}$.    We write   
\begin{equation}\label{PA1}
 \frac{\de   w}{\de z_N} (z, t)=\int_{\R^N}  \frac{\de \K_s}{\de z_N} (z-y, t) \t_E(y) dy= \int_{B_{\frac{r}{2}}(z)} \frac{\de \K_s}{\de y_N}(z-y,t) \t_E(y) dy+\int_{\R^N \setminus B_{\frac{r}{2}}(z)} \frac{\de \K_s}{\de y_N}(z-y, t) \t_E(y) dy.
\end{equation}
By a change of variable, \eqref{KernelFormula} and \eqref{KernelEstimate}, we have
\begin{align}\label{PA2}
\int_{\R^N \setminus B_{\frac{r}{2}}(z)} \frac{\de \K_s}{\de y_N}(z-y, t) dy&=t^{-\frac{1}{2s}} \int_{\R^N \setminus  t^{-\frac{1}{2s}}B_{{\frac{r}{2}}}(z)} \frac{\de P_s}{\de y_N}(t^{-\frac{1}{2s}}z-y) dy\nonumber\\
&= O\left(t^{-\frac{1}{2s}} \int_{\R^N \setminus t^{-\frac{1}{2s}} B_{{\frac{r}{2}} }(z)} {|t^{-\frac{1}{2s}}z -y|^{-N-1-2s}} dy\right)=O(t).
\end{align}
Integrating by parts, we have
\begin{align*}
\int_{B_{\frac{r}{2}}(z)} &\frac{\de \K_s}{\de y_N}(z-y, t) \t_E(y) dy= \int_{B_{\frac{r}{2}}(z) \cap E} \frac{\de \K_s}{\de y_N}(z-y, t) \t_E(y) dy-\int_{B_{\frac{r}{2}}(z) \cap E^c} \frac{\de \K_s}{\de y_N}(z-y, t) dy\\
&=2 \int_{B_{\frac{r}{2}}(z)\cap \de E } \K_s(z-y, t) e_N \cdot \nu_E(y) d\s(y)+\int_{\de B_r(z)} \K_s(z-y, t) e_N \cdot \nu_{B_{\frac{r}{2}}}(y) \t_E(y) d\s^\prime(y).
\end{align*}
By  a change of variable, \eqref{KernelFormula}, \eqref{KernelEstimate} and the fact that $Q_{r/8}\subset B_{r/4}\subset B_{\frac{r}{2}}(z) \subset Q_r$, we have
\begin{align}\label{PA3}
\int_{B_{\frac{r}{2}}(z) \cap \de E} \K_s(z-y, t) e_N \cdot \nu_E(y) d\s(y)&\geq C\int_{B_{r/8}^{N-1}} \K_s(z'-y^\prime, z_N-\gamma(y^\prime), t)dy^\prime \nonumber\\
 &=  Ct^{-\frac{1}{2s}} \int_{  B^{N-1}_{\frac{r}{8}t^{-\frac{1}{2s}}}} { P_s} (t^{-\frac{1}{2s}} z'-y',t^{-\frac{1}{2s}} z_N-t^{-\frac{1}{2s}}\g( t^{\frac{1}{2s}} y' )) dy' \nonumber\\
&\geq Ct^{-\frac{1}{2s}}\int_{B^{N-1}_{\frac{r}{8}t^{-\frac{1}{2s}}} \setminus B_2} \frac{dy^\prime}{1+|y^\prime|^{N+2s}},
\end{align}
provided $ r$. Next,  using  \eqref{KernelFormula-1} and recalling that $z\in B_{t^{\frac{1}{2s}}}$,
we then have
\begin{align*}
\left|\int_{\de B_{\frac{r}{2}}(z)} \K_s(z-y, t) e_N \cdot \nu_{B_{\frac{r}{2}}(z)}(y) d\s^\prime(y)\right|\leq \int_{\de B_{\frac{r}{2}}(z)} \K_s(z-y, t) d\s^\prime(y)&\leq t C \int_{\de B_{\frac{r}{2}}(z)} \frac{ 1}{|y|^{N+2s}} d\s^\prime(y) =O(t).
\end{align*}
From this and \eqref{PA3}, we deduce that 
$$
\int_{B_r(z)} \frac{\de \K_s}{\de y_N}(z-y, t) \t_E(y) dy \geq Ct^{-1/2s}.
$$
Combining this with \eqref{PA1} and  \eqref{PA2},  we get   
\begin{align*}
\frac{\de w}{\de z_N}(z, t)\geq Ct^{-1/2s}.
\end{align*}
Therefore \eqref{eq:nu_x} follows. Finally   \eqref{Partie2} follows   from the inverse function  theorem and the fact that $w$ is of class $C^1$ on $\R^N\times (0,\infty)$. 
\end{proof}

In the sequel, we will need the following lemmas to estimate some error terms.
\begin{lemma}\label{L1-Convergence}
	For $s\in(0,1)$, we let $E\subset\R^N$ be a set of class $\mathcal{C}^{1,\beta}$, for some $\beta>2s$, as in Section \ref{section4}.
	For $r>0$, we set
	$$
	J_{r}(t):= \int_{Q_r} \K_s(y, t) \t_E(y) dy \qquad \textrm{and}\qquad I_r(t)= \int_{\R^N \setminus Q_r} \left(t^{-1} \K_s(y,t)-\frac{C_{N,s}}{|y|^{N+2s}}\right) \t_E(y) dy.
	$$ 
	Then we have
	$$
	|J_r(t)|\leq C t r^{\b-2s} \qquad \textrm{and}\qquad \lim_{t \to 0} I_r(t)=0,
	$$
	where $C$ is a positive constant depending only on $N$, $\beta,$ $s$ and $E$.
\end{lemma}
\begin{proof}
	Since $\tau_E=\1_E-\1_{\R^N\setminus \ov E}$, we get
	\begin{align*}
	J_r(t)&=\int_{Q_r\cap E}\K_s(y,t)dy-\int_{Q_r\cap E^c}\K_s(y,{t})dy\\\\
	&=\int_{B_r^{N-1}}\int_{\gamma(y')}^r\K_s((y',y_N),{t})dy_N dy'-\int_{B_r^{N-1}}\int_{-r}^{\gamma(y')}\K_s((y',y_N),{t})dy_Ndy'\\\\
	&=\int_{B_r^{N-1}}\Biggl(\int_{\gamma(y')}^r\K_s((y',y_N),{t})dy_N-\int_{-r}^{\gamma(y')}\K_s((y',y_N),{t})dy_N\Biggr)dy'\\\\
	&=\int_{B_r^{N-1}}\Biggl(\int_{\gamma(y')}^r\K_s((y',y_N),{t})dy_N-\int_{-r}^{-\gamma(y')}\K_s((y',y_N),{t})dy_N'
	-\int_{-\gamma(y')}^{\gamma(y')}\K_s((y',y_N),{t})dy_N\Biggr)dy'\\\\
	&=\int_{B_r^{N-1}}\Biggl(\int_{\gamma(y')}^r\K_s((y',y_N),{t})dy_N+\int_{r}^{\gamma(y')}\K_s((y',-y_N),{t})dy_N
	-\int_{-\gamma(y')}^{\gamma(y')}\K_s((y',y_N),{t})dy_N\Biggr)dy'.
	\end{align*}
	Since the map $y \longmapsto K_s(y,{t})$ is radial, we have $\K_s(y',y_N,{t})=\K_s(y',-y_N,{t})$ so that
	$$
	\int_{\gamma(y')}^r\K_s((y',y_N),{t})dy_N+\int_{r}^{\gamma(y')}\K_s((y',-y_N),{t})dy_N=0.
	$$
	Therefore
	$$
	J_r(t)=-\int_{B_r^{N-1}}\int_{-\gamma(y')}^{\gamma(y')}\K_s((y',y_N),{t})dy_Ndy'=-2 \int_{B_r^{N-1}}\int_{0}^{\gamma(y')}\K_s((y',y_N),{t})dy_Ndy'.
	$$
	Then, by \eqref{KernelFormula-1}, 
	\begin{align*}
	|J_r(t)|&\leq 2C_{N,s} {t}\left|\int_{B_r^{N-1}}\int_{0}^{\gamma(y')}\frac{1}{|(y',y_N)|^{N+2s}} dy_Ndy'\right| \leq C {t} r^{\beta-2s}.
	\end{align*}
	where $C$ is a positive constant depending on $N$, $\b$ and $s$ and which may change from a line to another.  Next,   using \eqref{KernelFormula-1}, \eqref{ConvergenceUnifandLoc} and the    dominate convergence theorem, we obtain 
	$$
	\lim_{t \to 0} I_r(t)=0.
	$$
	This then ends the proof.
\end{proof}
%\begin{lemma}\label{LemmaComplement}
%We let $r>0$ and  $s\in (0,1)$. Then we have
%$$
%\int_{\R^N\setminus Q_r}\K_s(y-o_t(1) e_N,{t})\tau_E(y)dy=O({t}), \hspace{.5cm}\mbox{as $t\rightarrow 0.$}
%$$
%\end{lemma}
%\begin{proof}
%Since $|\t_E(y)|=1$, we then have
%$$
%\Bigg|\int_{\R^N\setminus Q_r}\K_s(y-o_t(1) e_N,{t})\tau_E(y)dy\Bigg|\leqslant\int_{\R^N \setminus B^N_r}\K_s(y-o_t(1) e_N,{t})dy\leq {t} \int_{\R^N \setminus B^N_r}\frac{dy}{|o_t(1) e_N-y|^{N+2s}}dy,
%$$
%where the last inequality is due to \eqref{ConvergenceUnifandLoc}. Moreover, using the fact that $\displaystyle\lim_{t \to 0} o_t(1)=0$, we then get
%$$
%\int_{\R^N \setminus B^N_r}\frac{dy}{|o_t(1) e_N-y|^{N+2s}}dy=\int_{\R^N \setminus B^N_r(o_t(1) e_N)}\frac{dy}{|y|^{N+2s}}dy \leq \int_{\R^N \setminus %B^N_{r/2}}\frac{dy}{|y|^{N+2s}}dy=O(1) \qquad \textrm{ as  $t \to 0$}.
%$$
%This then ends the proof of the lemma.
%\end{proof}

\begin{lemma}\label{DistanceEstimate}
	Let $s\in(0,1)$ and let $x=vte_N\in\partial E_t\cap B_{\s_s(t)^{\frac{1}{2s}}}$, with $E_t$ given by \eqref{eq:scaled-sets}. Then
\begin{align*}
vt=O\left((\sigma_s(t))^{\frac{1+2s}{2s}}\right) \qquad\textrm{ as $t\to 0$.}
\end{align*}
\end{lemma}

\begin{proof}
	Since $x=vte_N\in\partial E_t$, we have that  $u(x,\sigma_s(t))=0.$
	By the fundamental theorem of calculus, we have
	\begin{align*}
	u(x,\sigma_s(t))=u(0,\sigma_s(t))+vt\int_0^1\frac{\partial  u}{\partial x_N}(\theta vte_N,\sigma_s(t))d\theta=0
	\end{align*}
so that
	\begin{align}\label{AD1}
	vt\int_0^1\frac{\partial  u}{\partial x_N}(\theta vte_N,\sigma_s(t))d\theta=-u(0,\sigma_s(t)).
	\end{align}
We write 
$$
u(0,\sigma_s(t))=\int_{\R^N} \K_s(y,\sigma_s(t)) \t_E(y) dy= \int_{Q_r} \K_s(y,\sigma_s(t)) \t_E(y) dy+\int_{Q_r^c} \K_s(y,\sigma_s(t)) \t_E(y) dy.
$$
Then by  Lemma  \ref{L1-Convergence} and \eqref{KernelFormula-1}, we have
\begin{align}\label{AD3}
\left|\int_{Q_r} \K_s(y,\sigma_s(t)) \t_E(y) dy\right| \leq C \s_s(t) \qquad \textrm{ and } \qquad\int_{Q_r^c} \K_s(y,\sigma_s(t)) \t_E(y) dy=O(\sigma_s(t))
\end{align}
for some constant $C$ depending on $r$. Furthermore by \eqref{eq:nu_x}, we have
\begin{align}\label{AD4}
\frac{\partial  u}{\partial x_N}(\theta vte_N,\sigma_s(t))\geq C(\sigma_s(t))^{-1/2s}.
\end{align}
Therefore, the result immediately follows from \eqref{AD1}, \eqref{AD3} and \eqref{AD4}.  
\end{proof}
\begin{lemma}\label{lemmaComplement22}
Under the assumptions of Lemma \ref{DistanceEstimate},   we have
	$$
	\int_0^1\int_{B^{N-1}_r} \int_0^{\gamma(y^\prime)-vt} y_N \frac{\de \cK_s}{\de y_N}(y^\prime, \theta y_N, \sigma_s(t)) dy d\theta=O\left(\sigma_s(t)\right) \qquad \textrm{ as $t \to 0$.}
	$$
\end{lemma}
\begin{proof}
	Let $\theta \in [0, 1]$. By \eqref{KernelFormula},  \eqref{KernelEstimate} and a   change of variable, we have
	\begin{align*}
&	\int_{B^{N-1}_r} \int_0^{\gamma(y^\prime)-vt} y_N \frac{\de \cK_s}{\de y_N} (y^\prime, \theta y_N, \sigma_s(t)) dy\\
&=(\sigma_s(t))^{\frac{-N-1}{2s}}\int_{B^{N-1}_r} \int_0^{\gamma(y^\prime)-vt} y_N \frac{\de P_s}{\de y_N} (y'(\sigma_s(t))^{-\frac{1}{2s}}, \theta y_N(\sigma_s(t))^{-\frac{1}{2}}) dy\\
	&\leq C \int_{B^{N-1}_{r(\sigma_s(t))^{-1/2s}}} \int_0^{(\sigma_s(t))^{-1/2s}\left(\gamma((\sigma_s(t))^{1/2s}y^\prime)-vt\right)} \frac{y_N}{1+|y^\prime|^{N+2s+1}} dy\\
	&\leq C (\sigma_s(t))^{-1/s}\int_{B^{N-1}_{r(\sigma_s(t))^{-1/2s}}} \frac{\left(\gamma((\sigma_s(t))^{1/2s}y^\prime)-vt\right) ^2}{1+|y^\prime|^{N+2s+1}} dy\\&=O\left(\sigma_s(t)\right)+O\left(v^2t^2(\sigma_s(t))^{-1/s}\right)+O\left(vt\right).
	\end{align*}
	Applying   Lemma \ref{DistanceEstimate}, we get  
	$$\int_0^{\gamma(y^\prime)-vt} y_N \frac{\de \cK_s}{\de y_N}(y^\prime, \theta y_N, \sigma_s(t)) dy d\theta=O\left(\sigma_s(t)\right),
	$$
	as $t \to 0$.
	This then ends the proof.
\end{proof}
\section{Proof of Theorem \ref{MainResult} in the case $s\in(0,\frac{1}{2})$}\label{section1}
In this section, we start by the following preliminary result.
\begin{lemma}\label{Lemma2}
	Let $s \in (0, 1/2)$. We assume that $E$ is of class $\calC^{1,\b}$ for some $\b >2s$ satisfying \eqref{eq:deE-graph}. Then, for all $\theta \in [0,1]$,  we have
	\begin{equation}\label{Partie1}
	\int_{\R^N}\frac{\partial\K_s}{\partial y_N}(y^\prime, y_N-vt\theta,\sigma_s(t))\tau_E(y)dy=
	2 (\sigma_s(t))^{-1/2s} \int_{\R^{N-1}} P_s(y^\prime, 0)dy^\prime+O\left((\sigma_s(t))^{\frac{2s-1}{2s}}\right).
	\end{equation}
\end{lemma}
\begin{proof}
	We have
	\begin{align*}
	\int_{\R^N}\frac{\partial \K_s}{\partial y_N}(y^\prime, y_N-vt\theta,\sigma_s(t))\tau_E(y)dy&=\int_{B_r}\frac{\partial \K_s}{\partial y_N}(y^\prime, y_N-vt\theta,\sigma_s(t))\tau_E(y)dy\\\\
	&+\int_{B_r^c}\frac{\partial \K_s}{\partial y_N}(y^\prime, y_N-vt\theta,\sigma_s(t))\tau_E(y)dy,
	\end{align*}
	where $B_r$ is the ball of $\R^N$ centered at the origin and of radius $r>0$.
	By integration by parts, we have
	\begin{align*}
	\int_{B_r}\frac{\partial \K_s}{\partial y_N}(y^\prime, y_N-vt\theta,\sigma_s(t))\tau_E(y)dy&=2 \int_{\de E \cap B_r} \K_s(y^\prime, y_N-vt\theta,\sigma_s(t))\nu_N(y)d \sigma(y)\nonumber\\\\
	&+\int_{\de B_r} \K_s(y^\prime, y_N-vt\theta, \sigma_s(t))e_N\cdot\nu_{B_r}(y)\t_E(y) d \sigma^\prime(y).
	\end{align*}
	Therefore
	\begin{align}\label{l0}
	\int_{\R^N}&\frac{\partial \K_s}{\partial y_N}(y^\prime, y_N-vt\theta,\sigma_s(t))\tau_E(y)dy=2\int_{\partial E\cap B_r}\K_s(y^\prime, y_N-vt\theta,\sigma_s(t))\nu_N(y') d\sigma(y')\nonumber\\\
	&+\int_{B_r^c}\frac{\partial \K_s}{\partial y_N}(y^\prime, y_N-vt\theta,\sigma_s(t))\tau_E(y)dy+\int_{\de B_r} \K_s(y^\prime, y_N-vt\theta,\sigma_s(t))\frac{y_N}{r}\t_E(y) d \sigma^\prime(y).
	\end{align}
 Then by a change of variable  and \eqref{nu}, we have
	$$
	\int_{\partial E\cap B_r}\K_s(y^\prime, y_N-vt\theta,\sigma_s(t))\nu_N(y) d\sigma(y)=\int_{B_r^{N-1}}\K_s(y',\gamma(y')-vt\theta,\sigma_s(t)) dy'.
	$$
	By the Fundamental Theorem of calculus, we can write
	\begin{equation}\label{Expansionnn}
	\K_s(y',\gamma(y')-vt\theta,\sigma_s(t))=\K_s(y',0,\sigma_s(t))+\left(\gamma(y')-vt\theta\right)\int_0^1 \frac{\partial\K_s}{\partial y_N}(y',\theta^\prime\left(\gamma(y^\prime)-vt\theta)\right),\sigma_s(t)) d\theta^\prime.
	\end{equation}
	In the following, we let
	\begin{equation}\label{p0}
	\e(y^\prime):=\gamma(y^\prime)-vt\theta.
	\end{equation}
	Then we have
	\begin{align}
&	\int_{\partial E\cap B_r}\K_s(y^\prime, y_N-vt\theta,\sigma_s(t))\nu_N(y) d\sigma(y) =\int_{B_r^{N-1}}\K_s(y',\e(y^\prime),\sigma_s(t)) dy' \nonumber\\
	&=\int_{B_r^{N-1}}\K_s(y',0,\sigma_s(t)) dy'+\int_0^1 \int_{B_r^{N-1}} \e(y^\prime) \frac{\partial\K_s}{\partial y_N}(y',\theta^\prime\e(y^\prime),\sigma_s(t)) dy^\prime d\theta^\prime. \label{R0}
	\end{align}
	Therefore
	By a change of variable, \eqref{KernelFormula} and \eqref{KernelEstimate}, we have
	\begin{align}
	\int_{B_r^{N-1}}\K_s(y',0,\sigma_s(t))dy'&=\int_{B_r^{N-1}}(\sigma_s(t))^{\frac{-N}{2s}}P_s(y'(\sigma_s(t))^{-1/2s},0)dy' \nonumber\\
	&=(\sigma_s(t))^{-1/2s}\int_{\R^{N-1}}P_s(y',0)dy'+(\sigma_s(t))^{-1/2s}\int_{\R^{N-1}\setminus B_{r(\sigma_s(t))^{-1/2s}}^{N-1}}P_s(y',0)dy' \nonumber\\
%	&=(\sigma_s(t))^{-1/2s}\int_{\R^{N-1}}P_s(y',0)dy'+O\left((\sigma_s(t))^{-1/2s}\int_{r(\sigma_s(t))^{-1/2s}}^{+\infty}\rho^{N-2}\rho^{-N-2s}d\rho\right) \nonumber\\
%	&=(\sigma_s(t))^{-1/2s}\int_{\R^{N-1}}P_s(y',0)dy'+O\left((\sigma_s(t))^{-1/2s}\int_{r(\sigma_s(t))^{-1/2s}}^{+\infty}\rho^{-2-2s}d\rho\right) \nonumber\\
	&=(\sigma_s(t))^{-1/2s}\int_{\R^{N-1}}P_s(y',0)dy'+O\left((\sigma_s(t))^{-1/2s}(\sigma_s(t))^{\frac{1+2s}{2s}}\right). \label{R2}
	\end{align}
	By  a change of variable, \eqref{KernelEstimate} and \eqref{p0}, we have
	\begin{align*}
&	\int_{B_r^{N-1}} \e(y^\prime) \frac{\partial\K_s}{\partial y_N}(y',\theta^\prime\e(y^\prime),\sigma_s(t)) dy^\prime\\
&=(\sigma_s(t))^{-1/s}\int_{B_{r(\sigma_s(t))^{-1/2s}}^{N-1}} \e(y^\prime (\sigma_s(t))^{1/2s}) \frac{\partial P_s}{\partial y_N}(y',\theta^\prime (\sigma_s(t))^{-1/2s}\e(y^\prime (\sigma_s(t))^{1/2s})) dy^\prime.
	\end{align*}
	We use \eqref{ParametrisationNonLocal}, \eqref{p0} and  Lemma \ref{DistanceEstimate} to get
	\begin{align*}
	(\sigma_s(t))^{-1/s} \e(y^\prime (\sigma_s(t))^{1/2s})&= O\left(|y^\prime|^{1+\b} (\sigma_s(t))^{\frac{\b-1}{2s}}\right)-vt\theta (\sigma_s(t))^{-1/s}\\&=O\left(|y^\prime|^{1+\b} (\sigma_s(t))^{\frac{2s-1}{2s}}\right)+ O\left((\sigma_s(t))^{\frac{2s-1}{2s}}\right)\qquad \textrm{in } B_{r(\sigma_s(t))^{-1/2s}}^{N-1}.
	\end{align*}
	Then by \eqref{KernelEstimate}, we have
	\begin{align*}
&	\int_{B_r^{N-1}} \e(y^\prime) \frac{\partial\K_s}{\partial y_N}(y',\theta^\prime\e(y^\prime),\sigma_s(t)) dy^\prime \nonumber\\
	&=O\left((\sigma_s(t))^{\frac{2s-1}{2s}}\int_{B_{r(\sigma_s(t))^{-1/2s}}^{N-1}}|y'|^{1+\beta}\frac{\partial P_s}{\partial y_N}(y',\theta^\prime (\sigma_s(t))^{-1/2s}\e(y^\prime (\sigma_s(t))^{1/2s})) dy^\prime\right) \nonumber\\
	&+O\left((\sigma_s(t))^{\frac{2s-1}{2s}}\int_{B_{r(\sigma_s(t))^{-1/2s}}^{N-1}}\frac{\partial P_s}{\partial y_N}(y',\theta^\prime (\sigma_s(t))^{-1/2s}\e(y^\prime (\sigma_s(t))^{1/2s})) dy^\prime\right) \nonumber\\
	&=O\left((\sigma_s(t))^{\frac{2s-1}{2s}}\int_{\R^{N-1}}\frac{1+|y'|^{1+\beta}}{1+|y^\prime|^{N+2s+1}} dy^\prime\right) =O\left((\sigma_s(t))^{\frac{2s-1}{2s}}\right).
	\end{align*}

	Hence
	\begin{equation}\label{R4}
	\int_0^1 \int_{B_r^{N-1}} \e(y^\prime) \frac{\partial\K_s}{\partial y_N}(y',\theta^\prime\e(y^\prime),\sigma_s(t)) dy^\prime d\theta^\prime=O\left((\sigma_s(t))^{\frac{2s-1}{2s}}\right) \qquad \textrm{as } t\to 0.
	\end{equation}
	It follows from \eqref{R0}, \eqref{R2} and \eqref{R4} that
	\begin{align}\label{b1}
	\int_{\partial E\cap B_r}\K_s(y^\prime, y_N-vt\theta,\sigma_s(t))\nu_N(y) d\sigma(y)=(\sigma_s(t))^{-1/2s}\int_{\R^{N-1}}P_s(y',0)dy' +O\left((\sigma_s(t))^{\frac{2s-1}{2s}}\right)\qquad \textrm{as } t\to 0.
	\end{align}
	%
	%
	%
	%
	%
%	We recall that
%	$$ 
%	\K_s(y, \sigma_s(t))= (\sigma_s(t))^{-N/2s} P_s(y (\sigma_s(t))^{-1/2s}) \qquad \textrm{and} \qquad
%	|\nabla P_s(y)| \leq C_{N,s}\frac{|y|^{N-1+2s}}{(1+|y|^{N+2s})^2}.
%	$$
By a change of variable  and the fact that $|\t_E(y)|\leq 1$, we have
	\begin{align}\label{Elk}
	\int_{B_r^c}\frac{\de \K_s}{\de y_N}(y^\prime, y_N-vt\theta,\sigma_s(t))\t_E(y) d y&=O\left((\sigma_s(t))^{-1/2s}\int_{B_{r(\sigma_s(t))^{-1/2s}}^c}\frac{\de P_s}{\de y_N}(y^\prime, y_N-vt(\sigma_s(t))^{-1/2s}\theta) d y\right)\nonumber\\\
	&=O\left((\sigma_s(t))^{-1/2s}\int_{B_{r(\sigma_s(t))^{-1/2s}}^c}\frac{1}{1+|y|^{N+2s+1}}dy\right)=O(\sigma_s(t)).
	\end{align}
	We use \eqref{ConvergenceUnifandLoc} to get, as $t\to0$,
	\begin{align*}
	\left|\int_{\de B_r} \K_s(y^\prime,y_N-vt\theta,\sigma_s(t))\frac{y_N}{r}\t_E(y) d \sigma^\prime(y)\right|&\leq \int_{\de B_r} \K_s(y^\prime,y_N-vt\theta,\sigma_s(t)) d \sigma^\prime(y)\\\
	&\leq \sigma_s(t) \int_{\de B_r} \frac{C_{N,s}}{|(y^\prime,y_N-vt\theta)|^{N+2s}} d \sigma^\prime(y) =O(\sigma_s(t)).
	\end{align*}
	Therefore, the expansion \eqref{Partie1} follows immediately from \eqref{l0}, \eqref{b1}, \eqref{Elk} and the above estimate. This ends the proof.
\end{proof}
The following result completes the proof of Theorem \ref{MainResult}  in the case $s\in (0,1/2)$. 
\begin{proposition}\label{Prop01/2}
Under the assumptions of Lemma \ref{Lemma2},  we have
\begin{equation}\label{NormalvelocityFract}
v=a_{N,s} H_s(0)+o_t\left(1\right) \qquad \textrm{ as } t \to 0,
\end{equation}
where $H_s(0)$ is the fractional mean curvature of $\de E$ at the point $0$ and the positive constant $a_{N,s}$ is given by
$$
a_{N,s}=\frac{C_{N,s} }{\displaystyle2 \int_{\R^{N-1}} P_s\left(y^\prime,0\right) dy^\prime}.
$$
\end{proposition}
\begin{proof}
We put with $x=vte_N$ and we recall that 	 
	$$
	u(x,\sigma_s(t))= \int_{\R^N} \K_s(y-x, \sigma_s(t)) \tau_E(y) dy=0.
	$$ By  the fundamental theorem of calculus, we have
	$$
	\K_s(y-x,\sigma_s(t))=\K_s(y,\sigma_s(t))-vt\int_0^1 \frac{\partial\K_s}{\partial y_N}(y^\prime, y_N-vt\theta,\sigma_s(t)) d \theta.
	$$
	Then
	\begin{align}\label{Solution1}
	u(x,\sigma_s(t))=\ti J_r(t)+\sigma_s(t) \ti I_r(t)+\sigma_s(t)C_{N,s} \int_{\R^N\setminus Q_r} \frac{\t_E(y)}{|y|^{N+2s}} dy-
	vt\int_{\R^N} \int_0^1 \frac{\partial\K_s}{\partial y_N}(y^\prime, y_N-vt\theta,\sigma_s(t))\tau_E(y)d \theta dy,
	\end{align}
	where $\ti J_r(t)=J_r(\s_s(t))$, $\ti I_r(t)=I_r(\s_s(t))$, while $I_r(t)$ and $J_r(t)$ are given by Lemma \ref{L1-Convergence}. Moreover, by Lemma \ref{Lemma2}, we have
	$$
	\int_{\R^N}\frac{\partial\K_s}{\partial y_N}(y^\prime, y_N-vt\theta,\sigma_s(t))\tau_E(y)dy=2 (\sigma_s(t))^{-1/2s} \int_{\R^{N-1}} P_s(y^\prime, 0)dy^\prime+O\left((\sigma_s(t))^{\frac{2s-1}{2s}}\right).
	$$
	Therefore
	\begin{align}\label{Lemma2Erro}
	\int_0^1 \int_{\R^N}\frac{\partial\K_s}{\partial y_N}(y^\prime, y_N-vt\theta,\sigma_s(t))\tau_E(y) dy d \theta=2 (\sigma_s(t))^{-1/2s} \int_{\R^{N-1}} P_s(y^\prime, 0)dy^\prime+O\left((\sigma_s(t))^{\frac{2s-1}{2s}}\right)
	\qquad  \textrm{ as } t\to 0.
	\end{align}
	Putting  \eqref{Lemma2Erro}  in  \eqref{Solution1}, we obtain  that
	\begin{align*}
&	u(x,\sigma_s(t))=\ti J_r(t)+\sigma_s(t)\ti I_r(t)+\sigma_s(t)C_{N,s} \int_{\R^N\setminus Q_r} \frac{\t_E(y)}{|y|^{N+2s}} dy\\
	&-
	vt\left(2 (\sigma_s(t))^{-1/2s} \int_{\R^{N-1}} P_s(y^\prime, 0)dy^\prime+O\left((\sigma_s(t))^{\frac{2s-1}{2s}}\right)\right)\\
	&=\sigma_s(t)\bigg[(\sigma_s(t))^{-1}\ti J_r(t)+\ti I_r(t)+C_{N,s} \int_{\R^N\setminus Q_r} \frac{\t_E(y)}{|y|^{N+2s}} dy\ 
	-2vt (\sigma_s(t))^{\frac{-2s-1}{2s}} \int_{\R^{N-1}} P_s(y^\prime, 0)dy^\prime+O\left(\sigma_s(t)\right)\bigg].
	\end{align*}
Recalling that  $\sigma_s(t)=t^{\frac{2s}{1+2s}}$ and using the fact that $ u(x,\sigma_s(t))=0$, we have
	\begin{align*}\label{Solution2}
	0=t^{\frac{-2s}{1+2s}} \ti J_r(t)+ \ti I_r(t)+C_{N,s} \int_{\R^N\setminus Q_r} \frac{\t_E(y)}{|y|^{N+2s}} dy-2v \int_{\R^{N-1}} P_s(y^\prime, 0)dy^\prime+O(t^{\frac{2s}{1+2s}}) \qquad\textrm{ as $t \to 0$}.
	\end{align*}
As a consequence,
	\begin{equation*}
	\left|C_{N,s} H_s(0)-2v \int_{\R^{N-1}} P_s(y^\prime, 0)dy^\prime\right|\leq \left| C_{N,s} H_s(0)-C_{N,s} \int_{\R^N\setminus Q_r} \frac{\t_E(y)}{|y|^{N+2s}} dy\right|+|t^{\frac{-2s}{1+2s}}\ti  J_r(t)|+ \ti I_r(t)+O(t^{\frac{2s}{1+2s}}).
	\end{equation*}
	Therefore by Lemma \ref{L1-Convergence},  taking the limsups  as $ t\to 0$ and as $r \to 0$ respectively, we obtain
	$$
	\limsup_{t\to 0}\left|C_{N,s}  H_s(0)-2v \int_{\R^{N-1}} P_s(y^\prime, 0)dy^\prime\right|=0.
	$$
	Hence
	$$
	v=\frac{C_{N,s}  H_s(0)}{\displaystyle  2\int_{\R^{N-1}} P_s(y^\prime, 0)dy^\prime}+o_t(1) \qquad \textrm{ as $t \to 0.$}
	$$
	This then ends the proof.
\end{proof}

\section{Proof of Theorem \ref{MainResult} in the case $s\in(\frac{1}{2}, 1)$}\label{section2}
We have the following result.
\begin{proposition}\label{s-in-1/2-1}
We consider $E$ a hypersurface of class $C^3$ satisfying the condition in Section \ref{section4}.
For $s\in(1/2, 1),$ we have 
\begin{equation}\label{NormalvelocityClassic1/2-1}
v=c_{N,s} H(0)+O\left(t^{\frac{2s-1}{2}}\right),\hspace{1cm}\mbox{ as $t\rightarrow 0.$}
\end{equation}
where $H(0)$ is the normalized  mean curvature of $\de E$ at $0$ and the positive constant $c_{N,s}$ is given by
$$
c_{N, s}=\frac{\displaystyle\int_{\R^{N-1}}|y^\prime|^2  P_s(y^\prime,0) dy^\prime}{\displaystyle 2\int_{\R^{N-1}}  P_s(y^\prime,0) dy^\prime}.
$$
\end{proposition}
\begin{proof}
	We let  $x=vt e_N\in \de E_t$ and we expand
	\begin{equation}\label{P1}
	u(x,\sigma_s(t))=\int_{\R^N} \K_s(y-x,\sigma_s(t)) \t_E(y) dy= \int_{Q_r} \K_s(y-x,\sigma_s(t)) \t_E(y) dy+\int_{Q_r^c} \K_s(y-x,\sigma_s(t)) \t_E(y) dy,
	\end{equation}
	By  \eqref{KernelFormula-1} and Lemma \ref{DistanceEstimate}, we have
	$$
	\int_{Q_r^c} \K_s(y-x,\sigma_s(t)) \t_E(y) dy= O(\sigma_s(t)) \qquad \textrm{ as $t \to 0$}.
	$$
	Therefore
	\begin{align}
	u(x,\sigma_s(t))= \int_{Q_r} \K_s(y-x,\sigma_s(t)) \t_E(y) dy+O(\sigma_s(t)).\label{Lundi111}
	\end{align}
By a change of variable, the fact that $\t_E=\1_E(x)-\1_{\R^N\setminus \ov E}(x)$ and $x=vt e_N$, we have
	\begin{align*}
	\int_{Q_r} &\K_s(y-x,\sigma_s(t)) \t_E(y) dy=\int_{E\cap \B_r} \K_s(y-x,\sigma_s(t)) dy-\int_{E^c\cap Q_r} \K_s(y-x,\sigma_s(t))dy\\\\
	&=\int_{B_r^{N-1}} \int_{-r}^{\gamma(y^\prime)} \K_s(y-x,\sigma_s(t)) dy-\int_{B_r^{N-1}} \int^{r}_{\gamma(y^\prime)} \K_s(y-x,\sigma_s(t)) dy\\\\
	&=\int_{B_r^{N-1}} \int_{-r-vt}^{\gamma(y^\prime)-vt} \K_s(y,\sigma_s(t)) dy-\int_{B_r^{N-1}} \int^{r-vt}_{\gamma(y^\prime)-vt} \K_s(y,\sigma_s(t)) dy\\\\
	&=2\int_{B_r^{N-1}} \int_{0}^{\gamma(y^\prime)-vt} \K_s(y,\sigma_s(t)) dy+\int_{B_r^{N-1}} \int_{-r-vt}^0 \K_s(y,\sigma_s(t)) dy-\int_{B_r^{N-1}} \int^{r-vt}_0 \K_s(y,\sigma_s(t)) dy\\\\
	&=2\int_{B_r^{N-1}} \int_{0}^{\gamma(y^\prime)-vt} \K_s(y,\sigma_s(t)) dy+\int_{B_r^{N-1}} \int_{-r-vt}^{-r+vt} \K_s(y,\sigma_s(t)) dy.
	\end{align*}
The last line is due to the fact that the map $y_N \to \K_s(y, \sigma_s(t))$ is even  so that
$$
\int_{0}^{r-vt} \K_s(y,\sigma_s(t)) dy_N= -\int_{0}^{-r+vt} \K_s(y,\sigma_s(t)) dy_N.
$$
	Therefore we have
	\begin{equation}\label{ll1}
	\int_{Q_r} \K_s(y-x,\sigma_s(t)) \t_E(y) dy=2\int_{B_r^{N-1}} \int_{0}^{\gamma(y^\prime)-vt} \K_s(y,\sigma_s(t)) dy+\int_{B_r^{N-1}} \int_{-r-vt}^{-r+vt} \K_s(y,\sigma_s(t)) dy.
	\end{equation}
	By  \eqref{KernelFormula-1} and the fact that $vt=o_t(1)$, we have
	\begin{equation}
	\int_{B_r^{N-1}}\int_{-r-vt}^{-r+vt} \K_s(y,\sigma_s(t)) dy=O(\sigma_s(t)).
	\end{equation}
	By a change of variable, the Fundamental Theorem of Calculus, \eqref{KernelFormula} and \eqref{star}, we have
	\begin{align*}
	&\int_{B_r^{N-1}} \int_{0}^{\gamma(y^\prime)-vt} \K_s(y,\sigma_s(t)) dy\\
	&= \int_{B_r^{N-1}} \int_{0}^{\gamma(y^\prime)-vt}\K_s(y^\prime, 0,\sigma_s(t))dy+\int_{B_r^{N-1}} \int_{0}^{\gamma(y^\prime)-vt} \int_0^1 y_N  \frac{\de \K_s}{\de y_N}(y^\prime, \theta y_N, \sigma_s(t)) dyd\theta\\
	&= \int_{B_r^{N-1}}\K_s(y^\prime, 0,\sigma_s(t))\left(\frac{1}{2}\gamma_{y_i y_j}(0) y_i y_j+O(|y^\prime|^3)-vt\right)dy' \\
	&+\int_{B_r^{N-1}} \int_{0}^{\gamma(y^\prime)-vt} \int_0^1 y_N \frac{\de \K_s}{\de y_N}(y^\prime, \theta y_N, \sigma_s(t)) dyd\theta\\
	&= \frac{\D \gamma(0)}{2(N-1)}\int_{B_r^{N-1}}|y^\prime|^2 \K_s(y^\prime, 0,\sigma_s(t)) dy^\prime-vt\int_{B_r^{N-1}}\K_s(y^\prime, 0,\sigma_s(t)) dy^\prime\\\\
	&+\int_{B_r^{N-1}} \int_{0}^{\gamma(y^\prime)-vt} \int_0^1 y_N \frac{\de \K_s}{\de y_N}(y^\prime, \theta y_N, \sigma_s(t)) dyd\theta+O\left( \int_{B_r^{N-1}} |y^\prime|^3\K_s(y^\prime, 0,\sigma_s(t))dy'\right).
	\end{align*}
	Therefore, recalling \eqref{eq:def-mc},
	\begin{align}\label{k11}
	\int_{B_r^{N-1}} \int_{0}^{\gamma(y^\prime)-vt} \K_s(y,\sigma_s(t)) dy= &\frac{H(0)}{2}\int_{B_r^{N-1}}|y^\prime|^2 \K_s(y^\prime, 0,\sigma_s(t)) dy^\prime-vt\int_{B_r^{N-1}}\K_s(y^\prime, 0,\sigma_s(t)) dy^\prime \nonumber \\
	&+\int_{B_r^{N-1}} \int_{0}^{\gamma(y^\prime)-vt} \int_0^1 y_N \frac{\de \K_s}{\de y_N}(y^\prime, \theta y_N, \sigma_s(t)) dyd\theta \nonumber\\
	&+O\left( \int_{B_r^{N-1}} |y^\prime|^3\K_s(y^\prime, 0,\sigma_s(t))dy'\right).
	\end{align}
	By a change of variable  and \eqref{KernelFormula}, we have
	\begin{equation}\label{k12}
	\int_{B_r^{N-1}}|y^\prime|^2 \K_s(y^\prime, 0,\sigma_s(t)) dy^\prime= (\sigma_s(t))^{\frac{1}{2s}} \int_{B_{r (\sigma_s(t))^{-\frac{1}{2s}}}^{N-1}}|y^\prime|^2 P_s(y^\prime, 0) dy^\prime
	\end{equation}
	and
	\begin{align}\label{1212}
	\int_{B_r^{N-1}}\K_s(y^\prime, 0,\sigma_s(t)) dy^\prime =(\sigma_s(t))^{-\frac{1}{2s}} \int_{B_{r (\sigma_s(t))^{-\frac{1}{2s}}}^{N-1}} P_s(y^\prime, 0) dy^\prime.
	\end{align}
	Moreover by \eqref{KernelEstimate}, we get
	\begin{equation}\label{k13}
	(\sigma_s(t))^{\frac{1}{2s}} \int_{\R^{N-1}\setminus B_{r (\sigma_s(t))^{-\frac{1}{2s}}}^{N-1}}|y^\prime|^2 P_s(y^\prime, 0) dy^\prime+(\sigma_s(t))^{\frac{-1}{2s}} \int_{\R^{N-1}\setminus B_{r (\sigma_s(t))^{-\frac{1}{2s}}}^{N-1}} P_s(y^\prime, 0) dy^\prime=O(\sigma_s(t)) \qquad \textrm{as  $t \to 0$}
	\end{equation}
	and 
	\begin{equation}\label{k1313}
	\int_{B_r^{N-1}} |y^\prime|^3\K_s(y^\prime, 0,\sigma_s(t)) dy^\prime=O(\sigma_s(t)).
	\end{equation}
	By Lemma \ref{lemmaComplement22}, we get
	\begin{equation}\label{k14}
	\int_{B^{N-1}_r}\int_0^{\gamma(y')-vt}\int_0^1 y_N\frac{\partial \K_s}{\partial y_N}(y',\theta y_N,\sigma_s(t))dyd\theta=O\left(\sigma_s(t)\right).
	\end{equation}
	Combining \eqref{ll1}, \eqref{k11}, \eqref{k12}, \eqref{1212},  \eqref{k13} and \eqref{k14}, we obtain
	\begin{align}\label{P3}
	\int_{Q_r} \K_s(y-x,\sigma_s(t)) \t_E(y) dy=&(\sigma_s(t))^{\frac{1}{2s}}\frac{H(0)}{2}\int_{\R^{N-1}}|y^\prime|^2 P_s(y^\prime, 0) dy^\prime \nonumber\\
&\quad-vt(\sigma_s(t))^{\frac{-1}{2s}}\int_{\R^{N-1}}P_s(y^\prime, 0) dy^\prime +O\left(\sigma_s(t)\right).
	\end{align}
	By \eqref{Lundi111} and \eqref{P3}, we obtain
	\begin{align*}
	u(x,\sigma_s(t))&=(\sigma_s(t))^{1/2s} H(0)\int_{\R^{N-1}}|y^\prime|^2  P_s(y^\prime,0) dy^\prime-2vt(\sigma_s(t))^{-1/2s} \int_{\R^{N-1}}  P_s(y^\prime,0) dy^\prime+O\left(\sigma_s(t)\right)\\
	&=(\sigma_s(t))^{-1/2s}\bigg[(\sigma_s(t))^{1/s}H(0)\int_{\R^{N-1}}|y^\prime|^2  P_s(y^\prime,0) dy^\prime-2vt \int_{\R^{N-1}}  P_s(y^\prime,0) dy^\prime+O\left((\sigma_s(t))^{\frac{1+2s}{2s}}\right)\bigg].
	\end{align*}
	Since  $x=vt \nu  \in \de E_t$, we have $ u(x,\sigma_s(t))= 0.$
Now, from the definition of  $\sigma_s(t)=t^s$, we deduce that 
	\begin{align*}
	H(0)\int_{\R^{N-1}}|y^\prime|^2  P_s(y^\prime,0) dy^\prime-2v \int_{\R^{N-1}}  P_s(y^\prime,0) dy^\prime+O\left(t^{\frac{2s-1}{2}}\right)=0.
	\end{align*}
	Thus
	$$
	v=c_{N,s} H(0)+O(t^{\frac{2s-1}{2}}),
	$$
	where
	$$
	c_{N, s}=\frac{\displaystyle\int_{\R^{N-1}}|y^\prime|^2  P_s(y^\prime,0) dy^\prime}{\displaystyle 2\int_{\R^{N-1}}  P_s(y^\prime,0) dy^\prime}.
	$$
	This then ends the proof.
\end{proof}

\section{Proof of Theorem \ref{MainResult} in the case $s=\frac{1}{2}$}\label{section3}
%For this particular case, $s=1/2$, the kernel $\K_{1/2}$ is explicitly given by 
%
As usual, we consider   the function
$$
u(x, t)=\K_{1/2} (\cdot, t)\star\tau_E(x)
$$
and recall that
\begin{equation*}
E_t:= \lbrace x \in \R^N \quad :\quad u(x, \s_{1/2}(t))\geq 0 \rbrace.
\end{equation*}
To alleviate the notations, for the following of this section, we write $\sigma_{1/2}(t):=\sigma(t).$  
\begin{proposition}\label{propositionS=1/2}
	For $s=1/2,$ we have
	\begin{equation}\label{NormalvelocityClassic1/2}
	v=b_{N}(t) H(0)+O\left(\frac{1}{\log(\sigma_{1/2}(t))}\right) \qquad \textrm{as } t \to 0,
	\end{equation}
	where $H(0)$ is the mean curvature of $\de E$ at $0$.
	% and the positive constant $b_{N,1/2}$ is given by
%	$$
%	b_{N,1/2}=\frac{\o_{N-2} C_{N, 1/2}}{2\displaystyle\int_{\R^{N-1}}  P_{1/2}(y^\prime,0) dy^\prime}
%	$$
%	and $C_{N,1/2}$ is the positive constant appearing in \eqref{Kernel} and $\o_{N-2}=|S^{N-2}|$.
\end{proposition}

\begin{proof}
Recall that $x=vt\nu\to 0$ as $t\to0$, thanks to Lemma \ref{DistanceEstimate}.
We write
	\begin{equation}\label{P111}
	u(x,\s(t))=\int_{\R^N} \K_{1/2}(y-x,\s(t)) \t_E(y) dy= \int_{Q_r} \K_{1/2}(y-x, \s(t)) \t_E(y) dy+\int_{Q_r^c} \K_{1/2}(y-x,\s(t)) \t_E(y) dy,
	\end{equation}
	where $Q_r=B^{N-1}_r \times (-r, r)$. By  \eqref{KernelFormula-1}, we have
	$$
	\int_{Q_r^c} \K_{1/2}(y-x, \s(t)) \t_E(y) dy=O(\s(t)) \qquad \textrm{ as  $t\to 0$}.
	$$ 
	Then, we have
	\begin{align}
	u(x,\s(t))=\int_{Q_r} \K_{1/2}(y-x,\s(t)) \t_E(y) dy+O(\s(t)). \label{Lundi211}
	\end{align}
	By a change of variable   and \eqref{star}, we have
	\begin{align*}
&	\int_{Q_r} \K_{1/2}(y-x,\s(t)) \t_E(y) dy=\int_{E\cap \B_r} \K_{1/2}(y-x,\s(t)) dy-\int_{E^c\cap Q_r} \K_{1/2}(y-x,\s(t))dy\\
	&=\int_{B_r^{N-1}} \int_{-r}^{\gamma(y^\prime)} \K_{1/2}(y-x,\s(t)) dy-\int_{B_r^{N-1}} \int^{r}_{\gamma(y^\prime)} \K_{1/2}(y-x,\s(t)) dy\\
	&=\int_{B_r^{N-1}} \int_{-r-vt}^{\gamma(y^\prime)-vt} \K_{1/2}(y,\s(t)) dy-\int_{B_r^{N-1}} \int^{r-vt}_{\gamma(y^\prime)-vt} \K_{1/2}(y,\s(t)) dy\\
	&=2\int_{B_r^{N-1}} \int_{0}^{\gamma(y^\prime)-vt} \K_{1/2}(y,\s(t)) dy+\int_{B_r^{N-1}} \int_{-r-vt}^0 \K_{1/2}(y,\s(t)) dy -\int_{B_r^{N-1}} \int^{r-vt}_0 \K_{1/2}(y,\s(t)) dy \\
	&=2\int_{B_r^{N-1}} \int_{0}^{\gamma(y^\prime)-vt} \K_{1/2}(y,\s(t)) dy+\int_{B_r^{N-1}} \int_{-r-vt}^{-r+vt} \K_{1/2}(y,\s(t)) dy.
	\end{align*}
	The last line is due to the fact that the map $y_N \to \K_{1/2}(y, \sigma(t))$ is even  so that
	$$
	\int_{0}^{r-vt} \K_{1/2}(y,\s(t)) dy_N= -\int_{0}^{-r+vt} \K_{1/2}(y,\s(t)) dy_N.
	$$
	Therefore we have
	\begin{align}\label{ll11}
	\int_{Q_r} \K_{1/2}(y-x,\s(t)) \t_E(y) dy&=2\int_{B_r^{N-1}} \int_{0}^{\gamma(y^\prime)-vt} \K_{1/2}(y,\s(t)) dy +\int_{B_r^{N-1}} \int_{-r-vt}^{-r+vt} \K_{1/2}(y,\s(t)) dy.
	\end{align}
Using \eqref{KernelFormula-1},  we find that 
%
%	From the Lemma \ref{DistanceEstimate}, we obtain
	\begin{equation}\label{Errror11}
	\int_{B_r^{N-1}} \int_{-r-vt}^{-r+vt} \K_{1/2}(y,\sigma(t)) dy=O(\s(t)).
	\end{equation}
	By a change of variable, the fundamental theorem of calculus, \eqref{KernelFormula} and \eqref{star}, we have
	\begin{align*}
	&\int_{B_r^{N-1}} \int_{0}^{\gamma(y^\prime)-vt} \K_{1/2}(y,\s(t)) dy\\\\
	&= \int_{B_r^{N-1}} \int_{0}^{\gamma(y^\prime)-vt}\K_{1/2}(y^\prime, 0,\s(t))dy+\int_{B_r^{N-1}} \int_{0}^{\gamma(y^\prime)-vt} \int_0^1 y_N  \frac{\de \K_{1/2}}{\de y_N}(y^\prime, \theta y_N, \s(t)) dyd\theta\\\\
	&= \int_{B_r^{N-1}}\K_{1/2}(y^\prime, 0,\s(t))\left(\frac{1}{2}\gamma_{y_i y_j}(0) y_i y_j+O(|y^\prime|^3)-vt\right)dy'\\\\
	&+\int_{B_r^{N-1}} \int_{0}^{\gamma(y^\prime)-vt} \int_0^1 y_N \frac{\de \K_{1/2}}{\de y_N}(y^\prime, \theta y_N, \s(t)) dyd\theta\\\\
	&= \frac{\D \gamma(0)}{2(N-1)}\int_{B_r^{N-1}}|y^\prime|^2 \K_{1/2}(y^\prime, 0,\s(t)) dy^\prime-vt\int_{B_r^{N-1}}\K_{1/2}(y^\prime, 0,\s(t)) dy^\prime\\\\
	&+\int_{B_r^{N-1}} \int_{0}^{\gamma(y^\prime)-vt} \int_0^1 y_N \frac{\de \K_{1/2}}{\de y_N}(y^\prime, \theta y_N, \s(t)) dyd\theta+O\left( \int_{B_r^{N-1}} |y^\prime|^3\K_{1/2}(y^\prime, 0, \s(t))dy'\right).
	\end{align*}
	Therefore
	\begin{align}\label{k111}
	&\int_{B_r^{N-1}} \int_{0}^{\gamma(y^\prime)-vt} \K_{1/2}(y, \s(t)) dy=\frac{H(0)}{2}\int_{B_r^{N-1}}|y^\prime|^2 \K_{1/2}(y^\prime, 0, \s(t)) dy^\prime \nonumber \\\
	&-vt\int_{B_r^{N-1}}\K_{1/2}(y^\prime, 0, \s(t)) dy^\prime+\int_{B_r^{N-1}} \int_{0}^{\gamma(y^\prime)-vt} \int_0^1 y_N \frac{\de \K_{1/2}}{\de y_N}(y^\prime, \theta y_N, \s(t)) dyd\theta+O(\s(t)).
	\end{align}

By \eqref{KernelFormula}, \eqref{KernelEstimate} and a change of variable, we have
	\begin{align*}
	&\int_{B^{N-1}_r}\int_0^{\gamma(y')-vt} y_N\frac{\partial \K_{1/2}}{\partial y_N}(y',\theta y_N, \s(t))dy\\
	&=O\left(
	\int_{B^{N-1}_{r\s(t)^{-1}}}\int_0^{\s(t)^{-1}\left(\gamma(y'\s(t))-vt\right)} \frac{y_N}{(1+|(y^\prime, \theta y_N)|^2)^{\frac{N+2}{2}}} dy
	\right) = O\left(
	\int_{B^{N-1}_{r\s(t)^{-1}}} \frac{\s(t)^{-2}\left(\gamma(y'\s(t))-vt\right)^2}{(1+|y^\prime|^2)^{\frac{N+2}{2}}} dy'
	\right)\\
	&= O\left( \s^{-2}(t)
	\int_{B^{N-1}_{r\s(t)^{-1}}} \frac{\left(\s(t)^2|y^\prime|^2-vt\right)^2}{(1+|y^\prime|^2)^{\frac{N+2}{2}}} dy'
	\right)=O\left(\s(t)\right)+O(vt)+O\left(v^2t^2(\sigma(t))^{-2}\right).
	\end{align*}
Now  Lemma \ref{DistanceEstimate} yields $vt=O(\s(t)^2)$ and  thus   
	\begin{align}\label{k141}
	\int_{B^{N-1}_r}\int_0^{\gamma(y')-vt}\int_0^1 y_N\frac{\partial \K_{1/2}}{\partial y_N}(y',\theta y_N,\sigma(t))dyd\theta&=O\left(\s(t)\right).
	\end{align}
	We get from \eqref{Lundi211},  \eqref{ll11}, \eqref{Errror11}, \eqref{k111} and \eqref{k141} that
	\begin{align*}
	u(x,\s(t))&=\s(t) H(0) \int_{B^{N-1}_{r \s(t)^{-1}}}|y^\prime|^2  P_{1/2}(y^\prime,0) dy^\prime-2vt \int_{B^{N-1}_{r \s(t)^{-1}}}  P_{1/2}(y^\prime,0) dy^\prime+O(\s(t))\hspace{.2cm} \textrm{ as $t \to 0$}.
	\end{align*}
	Thanks to \eqref{KernelEstimate}, we have
	$$
	\int_{\R^{N-1} \setminus B^{N-1}_{r \s(t)^{-1}}}  P_{1/2}(y^\prime,0) dy^\prime=O(\s(t)^2) \quad \textrm{and} \quad  \int_{B^{N-1}_{\s(t)^{-1}} \setminus B^{N-1}_{r \s(t)^{-1}}} |y'|^2 P_{1/2}(y^\prime,0) dy^\prime=O(1) \qquad \textrm{as } t \to 0.
	$$
	This implies that
	\begin{align}\label{ko1}
	u(x,\s(t))&= \s(t)H(0)\int_{B^{N-1}_{ \s(t)^{-1}}}|y^\prime|^2  P_{1/2}(y^\prime,0) dy^\prime-2vt (\s(t))^{-1} \int_{\R^{N-1}}  P_{1/2}(y^\prime,0) dy^\prime+O(\s(t)).
	\end{align}
	Using polar coordinates and  \eqref{KernelEstimate},  we then have
	\begin{equation}\label{ko2}
	\int_{B^{N-1}_{\s(t)^{-1}}}|y^\prime|^2  P_{1/2}(y^\prime,0) dy^\prime \asymp \calC_{N, 1/2} \o_{N-2} \int_{0}^{1/\s(t)} \frac{m^N}{(1+m^2)^{\frac{N+1}{2}}} dm,
	\end{equation}
	where $\o_{N-2}:=|S^{N-2}|$.
	By the  change of variable $\rho=\frac{1}{m}$, we have
	\begin{align*}
	%\label{Seq0}
	\int_{0}^{1/\s(t)} \frac{m^N}{(1+m^2)^{\frac{N+1}{2}}} dm&=\int_{\s(t)}^{+\infty}\frac{1}{\rho\left(1+\rho^2\right)^{\frac{N+1}{2}}}d\rho\nonumber\\\
	&=\int_{\s(t)}^1\frac{1}{\rho\left(1+\rho^2\right)^{\frac{N+1}{2}}}d\rho+\int_1^{+\infty}\frac{1}{\rho\left(1+\rho^2\right)^{\frac{N+1}{2}}}d\rho \nonumber\\\
	&=\int_{\s(t)}^1\frac{1}{\rho\left(1+\rho^2\right)^{\frac{N+1}{2}}}d\rho+O(1)= -\log(\s(t))+O(1).
	\end{align*}
 Letting $b_{N}(t):=\frac{\displaystyle\int_{B^{N-1}_{\s(t)^{-1}}}|y^\prime|^2  P_{1/2}(y^\prime,0) dy^\prime }{-2\log(\s(t))\displaystyle\int_{\R^{N-1}}  P_{1/2}(y^\prime,0) dy^\prime}$,  by \eqref{ko1}, \eqref{ko2} and the above estimate, we obtain, as $t\to0$, 
\begin{align*}
u(x,\s(t))&=\s(t)H(0) \int_{B^{N-1}_{\s(t)^{-1}}}|y^\prime|^2  P_{1/2}(y^\prime,0) dy^\prime -2vt(\s(t))^{-1} \int_{\R^{N-1}}  P_{1/2}(y^\prime,0) dy^\prime+O(\s(t)) \\
&=(\sigma(t))^{-1}\bigg[\s^2(t)H(0) \int_{B^{N-1}_{\s(t)^{-1}}}|y^\prime|^2  P_{1/2}(y^\prime,0) dy^\prime-2vt\int_{\R^{N-1}}  P_{1/2}(y^\prime,0) dy^\prime+O(\s^2(t)) \bigg].
\end{align*}
	
%	\begin{align*}
%	u(x,\s(t))&=-\s(t)\log(\s(t))H(0) C_{N, 1/2} \o_{N-2}-2vt(\s(t))^{-1} \int_{\R^{N-1}}  P_{1/2}(y^\prime,0) dy^\prime+O(\s(t)) \\
	%
%	&=(\sigma(t))^{-1}\bigg[\s^2(t)|\log(\s(t))|H(0) C_{N, 1/2} \o_{N-2}-2vt\int_{\R^{N-1}}  P_{1/2}(y^\prime,0) dy^\prime+O(\s^2(t)) \bigg].
%	\end{align*}
	Since  $x=vt\nu \in \de E_t$, we have
	$
	u(x, \s(t))= 0
	$. 
	Recalling that  $t=\s^2(t)|\log(\s(t))|$, we finally get 
	$$
	0=\s(t)\log(\s(t))\left[H(0)\frac{\displaystyle\int_{B^{N-1}_{\s(t)^{-1}}}|y^\prime|^2  P_{1/2}(y^\prime,0) dy^\prime}{|\log(\s(t))|}-2v\int_{\R^{N-1}}  P_{1/2}(y^\prime,0) dy^\prime+O\left(\frac{1}{\log(\s(t))} \right)\right].
	$$
	Hence 
	$$
	v=b_{N}(t) H(0)+O\left(\frac{1}{\log(\s(t))} \right)\hspace{1cm}\mbox{ as $t\rightarrow 0.$}
	$$
%	where
%	$$
%	b_{N, 1/2}=\frac{\o_{N-2} C_{N, 1/2}}{2\displaystyle\int_{\R^{N-1}}  P_{1/2}(y^\prime,0) dy^\prime}.
%	$$
	The proof is then ended.
\end{proof}

\end{document}